\documentclass[12pt,a4paper,reqno]{amsart}
\usepackage{amsmath, amsthm, amsfonts, amssymb}
\usepackage{hyperref} 

\numberwithin{equation}{section}

\theoremstyle{plain}
\newtheorem{theorem}{Theorem}[section]
\newtheorem{proposition}[theorem]{Proposition}
\newtheorem*{prop-no-number}{Proposition}
\newtheorem{lemma}[theorem]{Lemma}
\newtheorem{corollary}[theorem]{Corollary}

\theoremstyle{definition}

\newtheorem{remark}[theorem]{Remark}

\renewcommand{\leq}{\leqslant}
\renewcommand{\geq}{\geqslant}
\renewcommand{\subset}{\subseteq}

\def\E{\mathbb{E}}
\def\Z{\mathbb{Z}}
\def\R{\mathbb{R}}

\def\C{\mathbb{C}}
\def\N{\mathbb{N}}
\def\P{\mathbb{P}}
\def\F{\mathbb{F}}

\newcommand{\ud}{\,\mathrm{d}}
\DeclareMathOperator{\Var}{\mathbf{Var}}

\DeclareMathOperator{\ssum}{\textstyle{\sum\nolimits}}
\newcommand{\Zmod}[1]{\Z_#1} 

\providecommand{\abs}[1]{\lvert#1\rvert}
\providecommand{\norm}[1]{\lVert#1\rVert}

\title[Finding almost-periods probabilistically]{A probabilistic technique for finding almost-periods of convolutions}

\author{Ernie Croot}
\address{School of Mathematics\\
     Georgia Institute of Technology\\
     Atlanta, GA 30332\\
     USA
}
\email{ecroot@math.gatech.edu}
\author{Olof Sisask}
\address{School of Mathematical Sciences\\
     Queen Mary, University of London\\
     Mile End Road\\
     London E1 4NS\\
     United Kingdom
}
\email{O.Sisask@qmul.ac.uk}
 

\begin{document}

\begin{abstract}
We introduce a new probabilistic technique for finding `almost-periods' of convolutions of subsets of groups. This gives results similar to the Bogolyubov-type estimates established by Fourier analysis on abelian groups but without the need for a nice Fourier transform to exist. We also present applications, some of which are new even in the abelian setting. These include a probabilistic proof of Roth's theorem on three-term arithmetic progressions and a proof of a variant of the Bourgain-Green theorem on the existence of long arithmetic progressions in sumsets $A+B$ that works with sparser subsets of $\{1, \ldots, N\}$ than previously possible.
In the non-abelian setting we exhibit analogues of the Bogolyubov-Freiman-Halberstam-Ruzsa-type results of additive combinatorics, showing that product sets $A_1 \cdot A_2 \cdot A_3$ and $A^2 \cdot A^{-2}$ are rather structured, in the sense that they contain very large iterated product sets. This is particularly so when the sets in question satisfy small-doubling conditions or high multiplicative energy conditions. We also present results on structures in $A \cdot B$.

Our results are `local' in nature, meaning that it is not necessary for the sets under consideration to be dense in the ambient group. In particular, our results apply to finite subsets of infinite groups provided they `interact nicely' with some other set.
\end{abstract}

\maketitle

\tableofcontents

\parindent 0mm
\parskip   4mm

\section{Introduction and statements of results}
There are many interesting problems that are concerned with counting various structures in subsets of groups. Many of these can be expressed in terms of the operation of convolution, defined for two functions $f, g : G \to \C$ on a group $G$ to be the function $f*g$ given by
\[ f*g(x) := \sum_{y \in G} f(y)g(y^{-1}x), \]
provided this exists for all $x \in G$. For example, many of the central objects of additive combinatorics can be expressed directly in terms of convolutions: the product set $A \cdot B = \{ ab : \text{$a \in A$, $b \in B$} \}$ of two subsets of a group is precisely the support of the function $1_A*1_B$, where $1_X$ denotes the indicator function of a set $X$, and the number of three-term arithmetic progressions in an additive set $A$, i.e., tuples $(a_1,a_2,a_3) \in A\times A\times A$ with $a_1+a_3=2a_2$, is $1_A*1_{-2\cdot A}*1_A(0)$. One may think of a convolution as being a sum of a function weighted by translates of another function and, as such, one may hope that convolutions are somewhat `smooth'. Indeed there are various senses in which this is true, and having precise notions of what it means can lead to interesting combinatorial consequences. Such results are often proved for abelian groups using the beautiful theory of Fourier analysis, where one uses the fact that convolutions and Fourier transforms interact in a very nice way. In this paper our aim is to demonstrate a new technique for establishing results about convolutions that are similar to those of Fourier analysis but that work on arbitrary groups, as well as to present applications.

\subsection{Notation}
Before we state our results let us introduce some notation---most of which is standard---directing the reader to the book \cite{tao-vu} of Tao and Vu or the paper \cite{tao:non-comm} of Tao for more details and interesting information about the concepts we use. Throughout the paper $G$ will denote a group (which may potentially be infinite). For two subsets $A$ and $B$ of $G$ we write $A \cdot B := \{ ab : \text{$a \in A$, $b \in B$} \}$ for the \emph{product set} of $A$ and $B$, and $A^{-1}$ for the collection of inverses of elements of $A$. Sometimes we shall omit the $\cdot$ and just juxtapose two sets to indicate the multiplication. For an element $t$ of $G$ we write $tA := \{ ta : a \in A \}$ for the left-translate of $A$ by $t$ and similarly for the right-translate $At$. If $k$ is a positive integer then we write $A^k := A \cdot A\cdots A$ for the $k$-fold product set of $A$, and $A^{-k}$ for the $k$-fold product set of $A^{-1}$.
For abelian groups we write the group operation additively and we give the corresponding definitions to $A+B$, $A-B$, $t+A$, $kA$, etc. The \emph{multiplicative energy} between two sets $A$ and $B$ is defined to be the quantity 
\[ E(A,B) := \sum_{x \in G} 1_A*1_B(x)^2; \]
for abelian groups this is known as the \emph{additive energy}.
For a function $f : G \to \C$ and a real number $p \geq 1$ we write $\norm{f}_p^p = \norm{f(x)}_p^p := \sum_{x \in G} |f(x)|^p$ for (the $p$th power of) the $L^p$ norm of $f$ provided this is finite. Thus $E(A,B) = \norm{1_A*1_B}_2^2$. A final piece of terminology: for finite groups $G$ we say that the \emph{density} of a set $A \subset G$ is $|A|/|G|$.

\subsection{The almost-periodicity results}
Our first result, then, is the following almost-periodicity-type theorem.

\begin{proposition}[$L^2$-almost-periodicity, local version]\label{L_2-local}
Let $G$ be a group, let $A, B \subset G$ be finite subsets, and let $\epsilon \in (0,1)$ be a parameter. Suppose $S \subset G$ is such that $|B \cdot S| \leq K|B|$. Then there is a set $T \subset S$ of size 
\[ |T| \geq \frac{|S|}{(2K)^{9/\epsilon^2}} \]
such that, for each $t \in T T^{-1}$,
\[ \norm{ 1_A*1_B(xt) - 1_A*1_B(x) }_2^2 \leq \epsilon^2 |A| |B|^2. \]
\end{proposition} 

The condition that there should be a set $S$ such that $|B \cdot S| \leq K|B|$ is what justifies the terminology `local': one does not need $B$ to be dense in its ambient group in order to apply the proposition effectively. All one needs is for $B$ to interact nicely with some large set $S$, a condition that we say more about in \S\ref{prelims}. If one knows little about the structure of $B$ one can still obtain useful conclusions from the proposition provided $B$ is dense in some structured set. For example, if $G = \Z$ and $B \subset [N] := \{1, \ldots, N\}$ with $|B| \geq \beta N$ (a case of interest in many problems) then one may take $S = [N]$ and $K = 2/\beta$. Similarly, if $G$ is finite then one can always take $S = G$, regardless of $B$, which immediately gives the following corollary.

\begin{corollary}[$L^2$-almost-periodicity, global version]\label{L_2-global}
Let $G$ be a finite group, let $A, B \subset G$, and let $\epsilon \in (0,1)$ be a parameter. Suppose $B$ has density $\beta$. Then there is a set $T \subset G$ of size at least $(\beta/2)^{9/\epsilon^2} |G|$ such that, for each $t \in T T^{-1}$,
\[ \norm{ 1_A*1_B(xt) - 1_A*1_B(x) }_2^2 \leq \epsilon^2 |A| |B|^2. \]
\end{corollary}

On an informal level these results say that convolutions are somewhat continuous: one may find a large number of translates $t$ such that the function $1_A*1_B$ does not change by much---in an $L^2$ sense---when translated by $t$. Having $L^2$-almost-periods provides one with good control in many applications, particularly those involving three-fold or higher convolutions, such as when dealing with the number of three-term progressions in a set or with a triple-fold product set $A\cdot B \cdot C$. But for certain applications involving only a single convolution it turns out that having $L^p$-almost-periods for a somewhat large $p$ is more useful.

\begin{proposition}[$L^p$-almost-periodicity, local version]\label{L_p-local}
Let $G$ be a group, let $A, B \subset G$ be finite subsets, and let $\epsilon \in (0,1)$ and $m \geq 1$ be parameters. Suppose $S \subset G$ is such that $|B \cdot S| \leq K|B|$. Then there is a set $T \subset S$ of size 
\[ |T| \geq \frac{|S|}{(2K)^{50m/\epsilon}} \]
such that
\[ \norm{ 1_A*1_B(xt) - 1_A*1_B(x) }_{2m}^{2m} \leq \max\left(\epsilon^m |A B| |B|^m, \norm{ 1_A*1_B }_m^m \right) \epsilon^m |B|^m \]
for each $t \in T T^{-1}$.
\end{proposition}

As before, this has the following `global' corollary.

\begin{corollary}[$L^p$-almost-periodicity, global version]\label{L_p-global}
Let $G$ be a finite group, let $A, B \subset G$ be subsets, and let $\epsilon \in (0,1)$ and $m \geq 1$ be parameters. Suppose $B$ has density $\beta$. Then there is a set $T \subset G$ of size at least $(\beta/2)^{50m/\epsilon} |G|$ such that
\[ \norm{ 1_A*1_B(xt) - 1_A*1_B(x) }_{2m}^{2m} \leq \max\left(\epsilon^m |A B| |B|^m, \norm{ 1_A*1_B }_m^m \right) \epsilon^m |B|^m \]
for each $t \in T T^{-1}$.
\end{corollary}

We give some further variants of the above propositions in \S\ref{main_proofs}. In particular one can with a slight change to the hypotheses find left-translates instead of right-translates, which may be more useful depending on the application.

Our proofs of the above propositions are of a probabilistic nature, involving a `random sampling' procedure that finds small subsets of one of the sets that behave similarly to the set itself (in a precise sense). This procedure is the same regardless of whether the group is commutative or not, which places our method in stark contrast to the Fourier-analytic methods that are typically the port of call for dealing with almost-periodicity in abelian groups. We say more about the abelian versions of the above results and the Fourier-analytic methods that lead to them in \S\ref{remarks}, turning now instead to applications of our results.

\subsection{Applications}\label{applications}
We shall apply the almost-periodicity results in four directions in this paper, namely towards
\begin{enumerate}
\item non-commutative analogues of the Bogolyubov-Freiman-Halberstam-Ruzsa theory that shows that sumsets are structured,
\item a low-density version of the Bourgain-Green theorem on long arithmetic progressions in sumsets $A+B$,
\item a probabilistic proof of Roth's theorem on arithmetic progressions and
\item a new result on the approximate translation-invariance of products of so-called strong $K$-approximate groups.
\end{enumerate}
We discuss each of these in turn.

\subsubsection*{Structures in product sets}
A general objective in additive combinatorics is to show that sumsets in abelian groups are rather structured objects. A rather useful such result due to Bogolyubov \cite{bogolyubov} that was highlighted by Ruzsa \cite{ruzsa:Freiman} in the additive-combinatorial context shows that sets $2A - 2A$ are highly structured, particularly if $A$ has small doubling. For non-abelian groups an analogue of this was recently proved by Sanders \cite{sanders:non-abelian-balog-szemeredi}:

\begin{theorem}\label{sanders-core}
Suppose $G$ is a group, $A \subset G$ is a finite set such that $|A^2| \leq K|A|$ and $k \in \N$ is a parameter. Then there is a symmetric set $S$ containing the identity such that
\[ S^k \subset A^2 \cdot A^{-2}  \text{ and }  |S| \geq \exp\left(-K^{O(k)}\right) |A|. \]
\end{theorem}

As noted in \cite{sanders:non-abelian-balog-szemeredi}, this is a variant of a result used in Tao's proof \cite{tao:solvable-freiman} of a Freiman-type theorem on the structure of sets with small doubling in solvable groups. Freiman-type results are ones that characterize subsets of groups that are group-like---more precisely subsets $A$ of a group $G$ that satisfy a small-doubling condition $|A^2| \leq K|A|$ or a small-tripling condition $|A^3| \leq K|A|$ for some fixed $K$---and there has been a concerted effort in recent years to try to establish such results in various classes of groups. In the commutative setting a rather precise and useful such characterization is provided by a theorem of Green and Ruzsa \cite{green-ruzsa:freiman} that generalizes a fundamental theorem of Freiman \cite{freiman}. In the non-commutative setting a number of interesting results have appeared recently \cite{breuillard-green1, breuillard-green2, breuillard-green-tao, FKP, hrushovski, pyber-szabo, tao:solvable-freiman}, though there is not yet a unified theory. Let us remark in the context of this paper, however, that results of the form of Theorem \ref{sanders-core} can be useful in proving such results: abelian results that find large Bohr sets in $2A-2A$ form a key step in many proofs of Freiman's theorem, and Theorem \ref{sanders-core} itself was recently used by Green, Sanders and Tao \cite{green-sanders-tao} to provide combinatorial proofs of some Freiman-type results of Hrushovski \cite{hrushovski}.

The almost-periodicity results of this paper are particularly well-suited to proving results of the form of Theorem \ref{sanders-core}, and doing so with reasonable bounds. Indeed, the following is a virtually immediate consequence of Proposition \ref{L_2-local}.

\begin{theorem}\label{core_set}
Suppose $G$ is a group, $A \subset G$ is a finite set such that $|A^2| \leq K|A|$ and $k \in \N$ is a parameter. Then there is a symmetric set $S \subset A^{-1} A$ containing the identity such that 
\[ S^k \subset A^2 \cdot A^{-2}  \text{ and }  |S| \geq \exp\left(-9k^2 K \log{2K} \right) |A|. \]
Furthermore, each element of $S^k$ has at least $|A|^3/2K$ representations as $a_1 a_2 a_3^{-1} a_4^{-1}$ with $a_i \in A$. 
\end{theorem}

Four-fold product sets of the above form are particularly pleasant to analyze, but it is not much harder to obtain a result that works with only triple product sets. To state this concisely it is convenient to introduce a small piece of non-standard terminology: for a triple $(A, B, C)$ of finite subsets of $G$ and an element $x \in G$ we shall say that $x$ is \emph{$\gamma$-popular} if $1_A*1_B*1_C(x) \geq \gamma (|A||B|)^{1/2} |C|$. That is, $x$ is $\gamma$-popular if it can be written as a product $abc$ with $a \in A$, $b \in B$ and $c \in C$ in at least $\gamma (|A||B|)^{1/2} |C|$ different ways. If $|A \cdot B \cdot C|$ is small then certainly there is a popular element, since
\[ |A||B||C| = \sum_{x \in A\cdot B \cdot C} 1_A*1_B*1_C(x) \leq |A\cdot B \cdot C|\ \sup_{x \in G} 1_A*1_B*1_C(x) \]
(see \S\ref{prelims}), but there are also much weaker conditions ensuring this.

\begin{theorem}\label{ABC}
Let $G$ be a group, let $A_1, A_2, A_3 \subset G$ be finite, non-empty sets and let $k \in \N$ be a parameter. Suppose $x$ is a $(1/K)$-popular element for $(A_1, A_2, A_3)$ and that there is a set $D \subset G$ such that $|A_3 \cdot D| \leq K'|A_3|$. Then there is a symmetric set $S \subset D D^{-1}$ containing the identity such that 
\[ x S^k \subset A_1  A_2  A_3 \text{ and }  |S| \geq \exp\left(-36 k^2 K^2 \log{2K'} \right) |D|. \]
\end{theorem}

In the abelian setting the non-local version of this result is in the same vein as a result of Freiman, Halberstam and Ruzsa \cite{freiman-halberstam-ruzsa} that finds long arithmetic progressions or Bohr sets in $A+A+A$ (see also \cite[Theorem 4.43]{tao-vu}); the best bounds currently known in this direction are due to Sanders \cite{sanders:longAPs}.

For the product of two sets the situation looks rather different, a phenomenon that has been observed in many different contexts. Whereas we cannot ensure that we can find a translate of a large iterated product set in $A \cdot B$, it turns out that we can always find a translate of any small subset of a large iterated product set.

\begin{theorem}\label{AB-struct}
Let $G$ be a group, let $A, B \subset G$ be finite, non-empty subsets and let $k, n \in \N$ be parameters. Suppose $|A \cdot B| \leq K|A|$ and $|B \cdot D| \leq K'|B|$. Then there is a symmetric set $S \subset D D^{-1}$ of size 
\[ |S| \geq \exp\left(-150 k^2 K \log{2K'} \log{2n} \right)|D| \]
such that the product set $A \cdot B$ contains a left-translate of any set $P \subset S^k$ of size at most $n$.
\end{theorem}

This theorem is a straightforward consequence of the $L^p$-almost-periodicity of $1_A*1_B$ given by Proposition \ref{L_p-local}. Our next application restricts this result to subsets of $\{1, \ldots, N\}$.

\subsubsection*{Arithmetic progressions in sumsets $A+B$}
Coupled with a `structure-generation' lemma that finds arithmetic progressions in iterated sumsets $kS$, Theorem \ref{AB-struct} quickly yields the following.

\begin{theorem}\label{A+B}
Let $N$ be a positive integer and let $A, B \subset [N]$ be non-empty sets of sizes $\alpha N$, $\beta N$. Then $A+B$ contains an arithmetic progression of length at least
\[ \tfrac{1}{2}\exp\left( c \left(\frac{\alpha \log{N}}{\log{4/\beta}} \right)^{1/4} \right), \]
where $c > 0$ is an absolute constant.
\end{theorem}

Results of this form have a rich history, starting with the paper \cite{bourgain:longAPs} of Bourgain. There it was shown, using a very insightful and sophisticated manipulation of sets of Fourier coefficients in the group $\Zmod{p}$, that if $A$ and $B$ are subsets of $[N]$ of densities $\alpha$ and $\beta$ then $A+B$ must contain an arithmetic progression of length at least
\begin{equation} \exp\left(c \left((\alpha \beta\log{N})^{1/3} - \log\log{N}\right)\right) \label{bourgain_bound} \end{equation}
for some absolute constant $c > 0$. This bound was improved by Green \cite{green:longAPs} using a different Fourier-analytic argument to the best bound that is currently known for high-density sets, increasing the exponent $1/3$ above to $1/2$; a similar bound has since also been established by Sanders \cite{sanders:longAPs} using another Fourier-analytic technique. By contrast, our result yields somewhat shorter arithmetic progressions for high-density sets (where $\alpha$ and $\beta$ are thought of as not depending on $N$) but is also able to deal with sets that are much smaller than previously possible. Whereas the previous bounds for the length of the arithmetic progressions one can find in $A+B$ are only non-trivial provided $\alpha \beta \geq C (\log \log N)^2/\log N$ for some absolute constant $C$, Theorem \ref{A+B} requires only $\alpha (\log{4/\beta})^{-1} \geq C/\log N$. Thus, whereas at least one of the sets had to have density at least $C \log \log N/(\log N)^{1/2}$ with previous bounds, the above theorem allows us to deal with pairs of sets each of which may have density as low as $C \log \log N/\log N$. In fact, one of the sets may have density as low as $\exp\left(-(\log N)^c \right)$, which illustrates a significant difference between our results and the Fourier-analytic ones. Our proof also adds another novelty: we are able to work directly in the group $\Z$, never needing to embed the sets in a group $\Zmod{p}$ (as is typical). We are also able to give a local version of the result; we present this and the proofs in \S\ref{longAPs}.

\subsubsection*{Roth's theorem}
Our next application concerns the quantity $r_3(N)$, the largest size of a subset of the integers $\{1,\ldots,N\}$ that is free from non-trivial three-term arithmetic progressions---that is, triples $(x,x+d,x+2d)$ with $d \neq 0$. As a consequence of our probabilistic proof of Proposition \ref{L_2-local} we are able to establish the following version of Roth's theorem \cite{roth} by completely combinatorial means.

\begin{theorem}\label{roths_theorem}
There is a function $\omega$ with $\omega(N) \to \infty$ as $N \to \infty$ such that 
\[ r_3(N) \leq \frac{N}{(\log\log{N})^{\omega(N)}} \]
for any positive integer $N$.
\end{theorem}

This bound for $r_3$ is marginally stronger than Roth's original $r_3(N) \ll N/\log\log{N}$, the beautiful Fourier-analytic proof of which has become a model argument in additive combinatorics. Subsequent Fourier-analytic arguments have demonstrated better bounds for $r_3$: the best bound currently known is due to Bourgain, who in \cite{bourgain} established that 
\[ r_3(N) \ll \frac{(\log\log N)^2}{(\log N)^{2/3}} N. \]
Roth's theorem has enjoyed many different proofs, including non-Fourier-analytic ones, and each new proof has typically offered a slightly different perspective on the problem. However, only the Fourier-analytic proofs seem to have given decent bounds: the methods that have not used Fourier analysis have generally been accompanied by tower-type bounds, establishing only that 
\[ r_3(N) \ll N /\log^*{N}; \]
see \cite[Chapter 10]{tao-vu} for references, as well as the more recent work \cite{polymath}. (The iterated logarithm of $N$, $\log^*{N}$, is defined to be the number of times it is necessary to take the logarithm of $N$ in order to get a number less than or equal to $1$, and thus grows extremely slowly.) It is therefore perhaps of interest that our method manages to give bounds of a similar quality to the Fourier-analytic proofs despite not using Fourier analysis. We give the proof of Theorem \ref{roths_theorem} in \S\ref{roth}.

\subsubsection*{Strong approximate groups}
We present one final application of the probabilistic technique: the following result says that products of certain `group-like' sets must have strong almost-periodicity properties.

\begin{proposition}\label{discontinuous}
Let $A$ be a finite subset of a group and let $\epsilon \in (0,1)$. Suppose $A$ has the property that every $x \in A^2$ has at least $|A|/K$ representations as $ab$ with $a, b \in A$. Then there is symmetric set $S \subset A^{-1}A$ of size
\[ |S| \geq \exp\left(-K^2 \log{2K} \log{8/\epsilon} \right)|A| \]
such that, for each $t \in S$,
\[ | tA^2 \bigtriangleup A^2 | \leq \epsilon|A^2|. \]
\end{proposition}

Green \cite{green:luczak-schoen} has suggested that one might call sets $A$ that satisfy the hypothesis of this proposition \emph{strong $K$-approximate groups}; that is, $A$ is a strong $K$-approximate group if $1_A*1_A(x) \geq |A|/K$ for each $x \in A^2$. Clearly any subgroup of a group is a strong $K$-approximate group with $K=1$, but there are more complex examples. For example, if $p > 3$ is a prime congruent to $3$ mod $4$ then the set $A \subset \Zmod{p}$ consisting of non-zero squares is a strong $(\frac{1}{2} - o(1))$-approximate group, for $A+A = \Zmod{p}\setminus \{0\}$ and $1_A*1_A(x) \geq (\frac{1}{2} - o(1))|A|$ for each non-zero $x \in \Zmod{p}$. Note also that if $A \subset G$ and $B \subset H$ are strong $K_A$- and $K_B$-approximate groups then $A \times B \subset G \times H$ is a strong $K_A K_B$-approximate group. We make some further remarks about strong approximate groups in \S\ref{discts}.

The remainder of this paper is laid out as follows. In the next section we describe some standard background material from the subject of arithmetic combinatorics. In \S\ref{main_proofs} we outline the basic idea behind our method and present the proofs of our almost-periodicity results. The proofs of the results on structures in product sets are very short and we give them immediately afterwards in \S\ref{non-abelian}. In \S\ref{structure-generation} we establish a structure-generation lemma that allows us to pass from arbitrary sets of translates in abelian groups to structured sets of translates. In \S\ref{longAPs} we give the proof of Theorem \ref{A+B} on arithmetic progressions in sumsets, and in \S\ref{roth} we present our proof of Roth's theorem. We present the proof of Proposition \ref{discontinuous} on strong approximate groups in \S\ref{discts}, and we close in \S\ref{remarks} with some further remarks, including a comparison with Fourier-analytic results.

\subsection{Acknowledgements}
We would like to thank Tom Sanders for many interesting and helpful conversations relating to several of the results of this paper. The second-named author is grateful for the support of an EPSRC Postdoctoral Fellowship, enjoyed while part of this work was carried out.

\section{Preliminaries on convolutions and product sets}\label{prelims}
In this section we record some useful standard results about convolutions and product sets; it may be largely skipped by those familiar with additive combinatorics. We follow Tao \cite{tao:non-comm}.

For functions on abelian groups the operation of convolution is commutative; this is not true in general for non-abelian groups. Convolution is, however, always bilinear and associative. A crucial link between convolutions and products is that the support of $1_{A_1} * \cdots * 1_{A_k}$ is the product set $A_1 \cdots A_k$. More precisely, 
\begin{equation}
1_{A_1} * \cdots * 1_{A_k}(x) = | \{ (a_1, \ldots, a_k) \in A_1 \times \cdots \times A_k : a_1 \cdots a_k = x \} |; \label{convolution-count} \end{equation}
convolutions thus count how many representations an element of a product of $k$ sets has a product of elements of the $k$ sets. For pairs of sets one also has the interpretation
\[ 1_A*1_B(x) = |A \cap xB^{-1}| = |B \cap A^{-1}x|. \]
For functions this change between left-translates and right-translates is illustrated by the reflection property
\begin{equation} \widetilde{f*g} = \widetilde{g}*\widetilde{f} \label{convolution-reflection} \end{equation}
where $\widetilde{f}(x) := f(x^{-1})$. Note that $\widetilde{1_X} = 1_{X^{-1}}$. Since convolutions are counts, sums of convolutions are also counts, and a full sum counts a particularly simple quantity:
\begin{equation} \sum_{x \in G} 1_{A_1} * \cdots * 1_{A_k}(x) = |A_1| \cdots |A_k|. \label{convolution_sum} \end{equation}

Many results in this paper involve conditions on the cardinalities of product sets. For two finite sets $A$ and $B$ in a group $G$, one always has the inequalities
\[ \max(|A|, |B|) \leq |A \cdot B| \leq |A||B| \]
with equality possible in various scenarios. Of course $|A \cdot B| \leq |G|$ as well. Of particular importance to this paper are the cases when the product set $A \cdot B$ is \emph{small}, though precisely what this means will depend on the context. Generally we shall say that $|A \cdot B|$ is small if it is at most $K|A|$ or $K|B|$ for some fixed number $K$, i.e., if it is within a constant factor of being as small as it could be. One generally thinks of a condition $|A \cdot B| \leq K|B|$ as showing that $A$ and $B$ share some structure, particularly if $A$ and $B$ are close in size. In particular this implies that $A$ and $B$ must themselves be somewhat structured, as follows from \cite[Lemma 3.2]{tao:non-comm}.

\begin{lemma}[Ruzsa triangle inequality]\label{triangle_inequality}
Let $A, B, C \subset G$ be finite, non-empty subsets of a group. Then
\[ |A \cdot C^{-1}| \leq \frac{|A \cdot B^{-1}| |B \cdot C^{-1}|}{|B|}. \]
\end{lemma}

Our almost-periodicity theorems are thus particularly effective when one of the sets $A$ and $B$ is structured in the sense of having \emph{small doubling} $|A^2| \leq K|A|$ or $|B^2| \leq K|B|$, or \emph{small differencing} $|A\cdot A^{-1}| \leq K|A|$ or $|B\cdot B^{-1}| \leq K|B|$, for some small, fixed $K$. In abelian groups the following result is particularly useful for bounding sizes of sumsets; see \cite[Chapter 6]{tao-vu} for references and a proof.

\begin{theorem}[Pl\"unnecke-Ruzsa inequality]\label{plunnecke-ruzsa}
Let $A$ and $B$ be finite subsets of an abelian group, and suppose $|A+B| \leq K|A|$. Then
\[ |n B - m B| \leq K^{m+n} |A| \]
for all integers $m, n \geq 1$.
\end{theorem}

As previously noted, however, one can substitute the above notion of structure for a much weaker one: that of being dense in a structured set (such as the ambient group). There are other ways in which one can weaken the notion of structure used; recall our definition of the multiplicative energy between two sets:
\[ E(A,B) = \sum_{x \in G} 1_A*1_B(x)^2. \]
If the product set $A \cdot B$ is small compared to either $A$ or $B$, in the sense that it has size at most $K|A|$ or $K|B|$, then $E(A,B)$ is large:
\begin{equation}
E(A,B) \geq \frac{1}{|A\cdot B|}\left(\sum_{x \in G} 1_A*1_B(x) \right)^2 = \frac{|A|^2|B|^2}{|A \cdot B|}, \label{energy-CS} \end{equation}
where the inequality follows from the Cauchy-Schwarz inequality. On the other hand, the condition $E(A,A) \geq |A|^3/K$ need not imply that $|A^2|$ is small, even in the abelian setting. We mention that there is a partial converse, however, that could be used in certain applications to keep the effectiveness of the bounds of this paper in the case when the sets in question have high multiplicative energy instead of small doubling: this is known as the Balog-Szemer\'edi-Gowers theorem. We point the interested reader to \cite[Chapter 2]{tao-vu} and \cite[Section 5]{tao:non-comm} for more information on this.

Many of the above properties have analogues for functions more general than indicator functions, of course. The distinction between indicator functions and more general functions tends not to be particularly important in practice; see the comments in \S\ref{remarks}.

\section{Proofs of the main propositions}\label{main_proofs}
Each of our propositions on almost-periodicity has to do with finding translates by which the convolution $1_A*1_B$ is approximately invariant in some norm. There are two basic ideas behind the proofs of these propositions. The first is that if one selects a small random subset $C \subset A$, then with high probability the convolution $1_C*1_B$ will approximate the function $\frac{|C|}{|A|}1_A*1_B$. This means that the approximation will hold for many subsets $C$ of $A$; so many, in fact, that there must be some relations amongst the sets: lots of them must in fact be \emph{translates} of one another, which is the second idea. The translates so obtained correspond to translates that leave $1_A*1_B$ approximately invariant (in the appropriate norm).

Surprisingly little background is needed to prove Proposition \ref{L_2-local}; all we shall assume is some basic familiarity with the probabilistic method---see for example \cite{alon-spencer} or \cite{tao-vu} for more details on this. We shall prove the following equivalent version of Proposition \ref{L_2-local}; the equivalence follows immediately from the reflection identity \eqref{convolution-reflection}.

\begin{proposition}[$L^2$-almost-periodicity, left-translates]\label{L_2-local-left}
Let $G$ be a group, let $A, B \subset G$ be finite subsets, and let $\epsilon \in (0,1)$ be a parameter. Suppose $S \subset G$ is such that $|S \cdot A| \leq K|A|$. Then there is a set $T \subset S^{-1}$ of size 
\[ |T| \geq \frac{|S|}{(2K)^{9/\epsilon^2}} \]
such that, for each $t \in T T^{-1}$,
\[ \norm{ 1_A*1_B(tx) - 1_A*1_B(x) }_2^2 \leq \epsilon^2 |A|^2 |B|. \]
\end{proposition} 
\begin{proof}
Let $k$ be an integer between $1$ and $|A|/2$ that we shall fix later and let $C$ be a random subset of $A$ of size $k$, chosen uniformly out of all such sets. Let us write $\mu_C := 1_C\cdot|A|/k$ for a normalized version of the indicator function of $C$. It is easy to see that $\E \mu_C*1_B(x) = 1_A*1_B(x)$ for each $x \in G$ and that the variance
\[ \Var(\mu_C*1_B(x)) = \E |\mu_C*1_B(x) - 1_A*1_B(x)|^2 \]
satisfies
\[ \Var(\mu_C*1_B(x)) \leq \tfrac{|A|}{k} 1_A*1_B(x). \]
Summing this inequality over all $x \in A \cdot B$, the support of $1_A*1_B$, we obtain
\begin{equation} \E \norm{ \mu_C*1_B(x) - 1_A*1_B(x) }_2^2 \leq |A|^2 |B|/k. \label{C-expectation}\end{equation}
Let us say that a set $C \in \binom{G}{k}$ \emph{approximates} $A$ if the bound
\[ \norm{ \mu_C*1_B(x) - 1_A*1_B(x) }_2^2 \leq 2|A|^2 |B|/k \]
holds. By \eqref{C-expectation} and Markov's inequality we thus have that
\begin{equation} \P_{C \in \binom{A}{k}}( \text{$C$ approximates $A$} ) \geq 1/2, \label{A-prob}\end{equation}
where $\P_{C \in \binom{X}{k}}$ refers to the uniform distribution on $k$-sets in a set $X$.

We now consider $k$-sets $C$ chosen uniformly at random from $Y := S \cdot A$ instead of $A$. Let $t \in S^{-1}$. Clearly
\begin{align*}
\P_{C \in \binom{Y}{k}}( \text{$tC$ approximates $A$} ) &= \P_{C \in \binom{tY}{k}}( \text{$C$ approximates $A$} ),
\end{align*}
and since $A \subset tY$ we see that this is at least
\[ \binom{|A|}{k} \binom{|S \cdot A|}{k}^{-1} \P_{C \in \binom{A}{k}}( \text{$C$ approximates $A$} ). \]
By \eqref{A-prob} and the hypothesis that $|S \cdot A| \leq K|A|$, then, we have that
\[ \P_{C \in \binom{Y}{k}}( \text{$tC$ approximates $A$} ) \geq \frac{1}{(2K)^k}. \]
Summing this inequality over all $t \in S^{-1}$ thus gives
\[ \E_{C \in \binom{Y}{k}} | \{ t \in S^{-1} : \text{$tC$ approximates $A$} \} | \geq \frac{|S|}{(2K)^k}. \]
In particular there exists a set $C$ for which the set 
\[ T := \{ t \in S^{-1} : \text{$tC$ approximates $A$} \} \]
has size at least $|S|/(2K)^k$. For this set $C$ we have 
\[ \norm{ \mu_C*1_B(x) - 1_A*1_B(tx) }_2^2 \leq 2|A|^2|B|/k \]
for each $t \in T$, whence
\[ \norm{ 1_A*1_B(tx) - 1_A*1_B(x) }_2^2 \leq 8|A|^2|B|/k \]
for each $t \in TT^{-1}$ by the triangle inequality. The proposition now follows upon choosing $k := \lceil 8/\epsilon^2 \rceil$. (Note that the conclusion of the proposition is trivial if $k > |A|/2$.)
\end{proof}

We need to argue only a little more subtly in order to establish the analogous estimate for higher $L^p$ norms: we just make use of higher moments than the variance. In order to do this we shall need some more information about random variables of the type $1_C*1_B(x)$ considered above. Since $1_C*1_B(x) = |C \cap xB^{-1}|$, a moment's thought reveals that this random variable follows a \emph{hypergeometric distribution}: a random variable $X$ is said to follow a hypergeometric distribution with parameters $N$, $M$ and $k$ if 
\[ \P(X = j) = \binom{M}{j}\binom{N-M}{k-j}\Big/\binom{N}{k} \]
for each integer $j \geq 0$. Thus one may think of $X$ as counting the number of marked objects one obtains when selecting $k$ objects randomly and without replacement from a population of $N$ objects, a total $M$ of which are marked. The proof of the following bounds on the moments of hypergeometrically distributed random variables is elementary, though somewhat tangential to our main arguments, so we postpone it till Appendix \ref{appendix:hyper_moments}.

\begin{lemma}\label{hyper_moments}
Let $m \geq 1$ and suppose that $X$ follows a hypergeometric distribution with parameters $N$, $M$ and $k$ as above. Then
\[ \E |X - \tfrac{kM}{N}|^{2m} \leq 2\left(3m \tfrac{k M}{N} + m^2 \right)^m. \]
\end{lemma}

With these estimates in hand the proof of Proposition \ref{L_p-local} is straightforward. Again we prove the following trivially equivalent version.

\begin{proposition}[$L^p$-almost-periodicity, left-translates]\label{L_p-local-left}
Let $G$ be a group, let $A, B \subset G$ be finite subsets, and let $\epsilon \in (0,1)$ and $m \geq 1$ be parameters. Suppose $S \subset G$ is such that $|S \cdot A| \leq K|A|$. Then there is a set $T \subset S^{-1}$ of size 
\[ |T| \geq \frac{|S|}{(2K)^{50m/\epsilon}} \]
such that
\[ \norm{ 1_A*1_B(tx) - 1_A*1_B(x) }_{2m}^{2m} \leq \max\left(\epsilon^m |A B| |A|^m, \norm{ 1_A*1_B }_m^m \right) \epsilon^m |A|^m \]
for each $t \in T T^{-1}$.
\end{proposition}
\begin{proof}
We follow the proof of Proposition \ref{L_2-local-left}, letting $C$ be a random subset of $A$ of size $k$ for some $k$ that is to be fixed. Fix an element $x \in G$. As alluded to above, the random variable $1_C*1_B(x)$ follows a hypergeometric distribution:
\[ \P(1_C*1_B(x) = j) = \binom{M}{j}\binom{|A|-M}{k-j}\Big/\binom{|A|}{k} \]
where $M := 1_A*1_B(x) = |A \cap xB^{-1}|$, the probability being nothing but the proportion of $k$-sets $C$ in $A$ that contain precisely $j$ elements from $A \cap xB^{-1}$. Lemma \ref{hyper_moments} therefore tells us that
\[ \E \big| 1_C*1_B(x) - \tfrac{k}{|A|} 1_A*1_B(x) \big|^{2m} \leq 2 \left(3mk\cdot 1_A*1_B(x)/|A| + m^2\right)^m \]
or, using the notation $\mu_C := \frac{|A|}{k} 1_C$,
\[ \E | \mu_C*1_B(x) - 1_A*1_B(x) |^{2m} \leq 2(m |A|/k)^m \left( 3\cdot 1_A*1_B(x) + m |A|/k \right)^m. \]
Summing over all $x \in A \cdot B$ then yields
\[ \E \norm{ \mu_C*1_B(x) - 1_A*1_B(x) }_{2m}^{2m} \leq 2(m |A|/k)^m  \sum_{x \in A \cdot B} \left(3\cdot 1_A*1_B(x) + m |A|/k \right)^m, \]
the right-hand side of which we denote by $\lambda$. From this it follows by Markov's inequality that
\[ \P\big( \norm{ \mu_C*1_B(x) - 1_A*1_B(x) }_{2m}^{2m} \leq 2\lambda \big) \geq 1/2. \]
We may now argue exactly as in the proof of Proposition \ref{L_2-local-left}, replacing the $L^2$-version of approximation there with this $L^{2m}$-version. We thus obtain a set $C \subset S \cdot A$ of size $k$ such that the set
\[ T := \{ t \in S^{-1} : \norm{ \mu_{C}*1_B(x) - 1_A*1_B(tx) }_{2m}^{2m} \leq 2\lambda \} \]
has size at least $|S|/(2K)^k$. The result now follows from the triangle inequality upon noting the bound
\[ \lambda \leq 2(m|A|/k)^m 3.05^m\max\left( \norm{ 1_A*1_B }_m, 20 m |A B| |A|/k \right)^m \]
and choosing $k := \left\lceil 49m/\epsilon \right\rceil$. 
\end{proof}

\begin{remark}
We have not attempted to optimize the constant $50$ that appears in the exponent of the density of the set $T$ in this proposition; one can certainly reduce it, though any such reduction would be largely irrelevant for our applications.
\end{remark}

\section{Structures in product sets}\label{non-abelian}
In this section we provide proofs of the applications discussed in the first part of \S\ref{applications}. These results were all versions of the statement that product sets are structured objects, with various meanings. 
Theorem \ref{core_set} said that sets $A^2 \cdot A^{-2}$ are structured in the sense that they contain large iterated product sets; this is perhaps the most straightforward consequence of Proposition \ref{L_2-local-left}: 

\begin{proof}[Proof of Theorem \ref{core_set}]
Set $\epsilon := 1/k\sqrt{K}$ and apply Proposition \ref{L_2-local-left} to $A$ with $B=S=A$ to obtain a set $T \subset A^{-1}$ of size at least $|A|/(2K)^{9k^2K}$ such that 
\[ \norm{ 1_A*1_A(tx) - 1_A*1_A(x) }_2^2 \leq \epsilon^2 |A|^3 \]
for each $t \in TT^{-1}$. Write $S := TT^{-1}$. By the triangle inequality we then have 
\[ \norm{ 1_A*1_A(tx) - 1_A*1_A(x) }_2^2 \leq |A|^3/K \]
for each $t \in S^k$. The left-hand side of this inequality can be expanded as
\begin{align*}
&2 \sum_{x \in G} 1_A*1_A(x)^2 - 2\sum_{x \in G} 1_A*1_A(tx)1_A*1_A(x) \\
&\quad = 2\left( E(A,A) - 1_A*1_A*1_{A^{-1}}*1_{A^{-1}}(t) \right).
\end{align*}
Since $A$ has small doubling, it also has large multiplicative energy by \eqref{energy-CS}: $E(A,A) \geq |A|^3/K$. Hence 
\[ 1_A*1_A*1_{A^{-1}}*1_{A^{-1}}(t) \geq |A|^3 \left(1/K - 1/2K\right) \geq |A|^3/2K. \]
Since $1_A*1_A*1_{A^{-1}}*1_{A^{-1}}$ has support $A^2 \cdot A^{-2}$, we thus have that $S^k \subset A^2 \cdot A^{-2}$ as desired. Furthermore, each element $t \in S^k$ has many representations as products in the way claimed, as follows from \eqref{convolution-count}.
\end{proof}

We record the following more general version of Theorem \ref{core_set}; the proof is the same except we do not specialize all the parameters when applying Proposition \ref{L_2-local-left}.

\begin{theorem}\label{ABBA}
Let $G$ be a group, let $A, B \subset G$ be finite, non-empty subsets and let $k \in \N$ be a parameter. Suppose $E(A,B) \geq |A|^2|B|/K$ and that $|D \cdot A| \leq K'|A|$ for some set $D \subset G$. Then there is a symmetric set $S \subset D^{-1} D$ containing the identity such that 
\[ S^k \subset A\cdot B \cdot B^{-1} \cdot A^{-1} \text{ and }  |S| \geq \exp\left(-9k^2 K \log{2K'} \right) |D|. \]
Furthermore, each element of $S^k$ has at least $|A|^2 |B|/2K$ representations as $a_1 b_1 b_2^{-1} a_2^{-1}$ with $a_i \in A$, $b_i \in B$.
\end{theorem}

Note that this really does generalize Theorem \ref{core_set} by \eqref{energy-CS}.

Theorem \ref{ABC} dealt with the product of three sets under the assumption of the existence of a `popular element'. Note that there are various conditions that will ensure the existence of a popular element for a triple of sets $(A,B,C)$: $A \cdot B \cdot C$ being small will certainly do, as will $\norm{1_A*1_B*1_C}_2$ being large. The condition $E(A,B) \geq |A|^2|B|/K$ is also a popularity-type condition, $E(A,B)$ equalling $1_A*1_B*1_{B^{-1}}*1_{A^{-1}}(1)$, and the pigeonhole principle shows that if the multiplicative energy $E(A,B)$ is large then there is a popular element for the triple $(B, B^{-1}, A^{-1})$.

\begin{proof}[Proof of Theorem \ref{ABC}]
Recall that we are given three finite sets $A_1$, $A_2$ and $A_3$, a `popular' element $x$ such that
\[ 1_{A_1}*1_{A_2}*1_{A_3}(x) \geq (|A_1||A_2|)^{1/2}|A_3|/K, \]
and a set $D$ such that $|A_3 \cdot D| \leq K'|A_3|$. Apply Proposition \ref{L_2-local} to the sets $A = A_2$ and $B = A_3$ with $\epsilon := 1/2kK$ to obtain a set $T \subset D$ of size at least $|D|/(2K')^{36k^2 K^2}$ such that
\[ \norm{ 1_{A_2}*1_{A_3}(yt) - 1_{A_2}*1_{A_3}(y) }_2^2 \leq \epsilon^2 |A_2||A_3|^2 \]
for each $t \in S := TT^{-1}$. Thus for each $t \in S^k$ we have
\[ \norm{ 1_{A_2}*1_{A_3}(yt) - 1_{A_2}*1_{A_3}(y) }_2^2 \leq |A_2||A_3|^2/4K^2. \]
Let $t \in S^k$. Then
\begin{align*}
&| 1_{A_1}*1_{A_2}*1_{A_3}(xt) - 1_{A_1}*1_{A_2}*1_{A_3}(x) | \\
&\qquad= \left| \sum_{y \in G} 1_{A_1}(y) \left( 1_{A_2}*1_{A_3}(y^{-1}xt) - 1_{A_2}*1_{A_3}(y^{-1}x) \right) \right| \\
&\qquad\leq |A_1|^{1/2} \norm{ 1_{A_2}*1_{A_3}(yt) - 1_{A_2}*1_{A_3}(y) }_2,
\end{align*}
the inequality being an application of the Cauchy-Schwarz inequality. Thus
\[ | 1_{A_1}*1_{A_2}*1_{A_3}(xt) - 1_{A_1}*1_{A_2}*1_{A_3}(x) | \leq (|A_1||A_2|)^{1/2}|A_3|/2K. \]
Since $x$ is a $(1/K)$-popular element and $t \in S^k$ was arbitrary, this completes the proof.
\end{proof}

We turn now to the case of two sets. Theorem \ref{AB-struct} is a special case of the following result, which has the advantage of giving stronger results in the situation when $|A \cdot B|$ is not small but the multiplicative energy $E(A,B)$ is still large.

\begin{theorem}\label{AB-struct-full}
Let $G$ be a group, let $A,B \subset G$ be finite, non-empty subsets and let $k, n \in \N$ be parameters. Suppose that
\begin{enumerate}
\item $E(A,B) \geq |A||B|^2/K_1$,
\item $|A \cdot B| \leq K_2 |A|$ and
\item $|B \cdot S| \leq K_3 |B|$.
\end{enumerate}
Then there is a set $T \subset S$ of size 
\[ |T| \geq \exp\left(-150 k^2 (K_1 K_2)^{1/2} (\log{2K_3}) (\log{2n}) \right)|S| \]
such that the product set $A \cdot B$ contains a left-translate of any set $P \subset (T T^{-1})^k$ of size at most $n$.
\end{theorem}
\begin{proof}
We may assume that $n \geq 2$. Set $m := \log{2n}$, define $\gamma$ by requiring $\norm{ 1_A*1_B }_m^m = \gamma^m |A B| |B|^m$ and set $\epsilon := \gamma/e k^2$. Applying Proposition \ref{L_p-local} to $A$ and $B$ with these parameters gives us a set $T \subset S$ with
\[ |T| \geq \frac{|S|}{(2K_3)^{50e k^2 (\log{2n})/\gamma}} \]
such that
\[ \norm{ 1_A*1_B(xt) - 1_A*1_B(x) }_{2m}^{2m} \leq \epsilon^m |B|^m \norm{ 1_A*1_B }_m^m \]
for each $t \in T T^{-1}$. Let $P \subset (T T^{-1})^k$ be a set of size at most $n$. Suppose for a contradiction that $A \cdot B$ does not contain a left-translate of $P$. Then for every $x \in G$ there must be an element $t \in P$ for which $xt \notin A\cdot B$, i.e., for which $1_A*1_B(xt) = 0$. Hence
\begin{align*}
n k^{2m} \epsilon^m |B|^m \norm{ 1_A*1_B }_m^m &\geq \sum_{t \in P} \norm{ 1_A*1_B(xt) - 1_A*1_B(x) }_{2m}^{2m} \\
&\geq \sum_{x \in G} 1_A*1_B(x)^{2m}.
\end{align*}
By the Cauchy-Schwarz inequality this is at least $\norm{ 1_A*1_B }_m^{2m}/|A B|$. Recalling the definition of $\epsilon$ and $m$ then gives the desired contradiction; hence there must be some element $x$ for which $xP \subset A \cdot B$.
The result now follows upon noting that $\gamma \geq 1/(K_1 K_2)^{1/2}$; this follows from H\"older's inequality and \eqref{convolution_sum}.
\end{proof}
\begin{remark}
The constant $150$ in the conclusion should not be taken seriously; it can obviously be improved.
\end{remark}
\begin{remark}
If $|A \cdot B|$ is not small compared to $|A|$ then the conclusion of the theorem becomes much less effective. If one still has the energy condition $E(A,B) \geq |A||B|^2/K$ and $A$ and $B$ are of a similar size then one can use the Balog-Szemer\'edi-Gowers theorem \cite[Theorem 5.2]{tao:non-comm} to obtain large subsets $A' \subset A$ and $B' \subset B$ that one can then apply the theorem to effectively; this would yield better bounds than using a large value of $K_2$ directly. We omit the details.
\end{remark}

\section{Obtaining structured sets of translates}\label{structure-generation}
While Propositions \ref{L_2-local} and \ref{L_p-local} yield very large sets of translates for which $1_A*1_B$ is approximately translation-invariant, one often needs these sets to be structured as well. Indeed, for the abelian applications in this paper we shall need to find \emph{arithmetic progressions} of such translates. With Fourier-analytic methods the existence of an arithmetic progression of almost-periods is usually easy to obtain since one usually gets a Bohr set of translates, but we do not have this convenience. Instead we shall generate the structure by repeated set-addition.

We say that an arithmetic progression $P$ in an abelian group has length $k$ if it can be written as 
\[ P = \{ a, a+d, \ldots, a+(k-1)d \} \]
for some non-zero element $d$. Note that this notion may be somewhat degenerate in some groups.

\begin{lemma}\label{local-sarkozy}
Let $G$ be an abelian group, let $S \subset G$ be a finite subset that satisfies $|S+S| \leq K|S|$ or $|S-S| \leq K|S|$ and let $k \in \N$. Suppose $A \subset S$ satisfies $|A| \geq \delta |S|$ where
\[ \delta > K^{3k/2}/|S|^{1/(k+1)}. \]
Then the set $kA - kA$ contains a symmetric arithmetic progression of length at least $2^{k+1}$ passing through $0$, with non-zero step $d \in A-A$.
\end{lemma}

To prove this we require a simple preliminary result. (This is not required if $K = 1$, as would be the case if $X$ is a group.)

\begin{lemma}\label{dilate-lemma}
Let $G$ be an abelian group and let $A \subset G$ be a finite subset satisfying $|A+A| \leq K|A|$ or $|A-A| \leq K|A|$. Let $k \in \N$. Then
\[ |A - 2^k \cdot A| \leq K^{3k}|A|. \]
\end{lemma}

Here we write $\lambda \cdot A$ for the dilate $\{ \lambda a : a \in A \}$. This result is Theorem 15 of Bukh \cite{bukh:dilates} specialized to the case $\lambda = -2$. We include the short proof for completeness.
\begin{proof}
By the Ruzsa triangle inequality, Lemma \ref{triangle_inequality}, we have that
\[ |A - 2^k \cdot A| \leq \frac{|A - 2 \cdot A| |2 \cdot A - 2^k \cdot A|}{|A|} \leq \frac{|A - 2 \cdot A| |A - 2^{k-1} \cdot A|}{|A|}. \]
Hence
\[ \frac{|A - 2^k \cdot A|}{|A|} \leq \left( \frac{|A - 2\cdot A|}{|A|} \right)^k. \]
The lemma then follows from the instance $|A-2\cdot A| \leq K^3|A|$ of the Pl\"unnecke-Ruzsa inequality, Theorem \ref{plunnecke-ruzsa}.
\end{proof}

\begin{proof}[Proof of Lemma \ref{local-sarkozy}]
It suffices to show that there are distinct elements $a,b \in A$ for which
\[ 2^j(a-b) \in A-A \text{ for each $j = 1,\ldots,k$}, \]
for then $[-2^k, 2^k]\cdot (a-b) \subset kA - kA$ by binary expansion. Upon rearranging, this is equivalent to there being distinct $a,b \in A$ and elements $x_j, y_j \in A$ for which
\begin{align*}
x_1 - 2a &= y_1 - 2b \\
x_2 - 4a &= y_2 - 4b \\
&\ \vdots \\
x_k - 2^k a  &= y_k - 2^k b.
\end{align*}
We claim that this system of equations must have a solution with $a \neq b$. Indeed, for each $(k+1)$-tuple of elements $(a,x_1,\ldots,x_k) \in A^{k+1}$, set
\[ f(a,x_1,\ldots,x_k) := (x_1-2a, x_2-4a, \ldots, x_k-2^k a). \]
The image of this function is a subset of $(S-2\cdot S)\times \cdots \times (S-2^k \cdot S)$, which by Lemma \ref{dilate-lemma} has size at most $K^{3k(k+1)/2} |S|^k$. So if $|A|^{k+1} > K^{3k(k+1)/2} |S|^k$, which is the case given our bound on $\delta$, then there must be two distinct tuples $\mathbf{a} = (a,x_1,\ldots,x_k)$ and $\mathbf{b} = (b,y_1,\ldots,y_k)$ in $A^{k+1}$ for which $f(\mathbf{a}) = f(\mathbf{b})$. Clearly such tuples must have $a \neq b$ and so provide a non-trivial solution to our system.
\end{proof}

\begin{remark}
One may wish to generate different types of structure depending on the group; for example, for problems in $\F_3^n$ it is more natural to generate subspaces instead of arithmetic progressions. Establishing such a result in $\F_3^n$ is relatively straightforward: it is easy to see that adding a symmetric subset of $\F_3^n$ to itself generates a subspace of dimension equal to the number of summands provided the set has enough linearly independent vectors.
\end{remark}

The proof of Lemma \ref{local-sarkozy} should be compared with an argument of the first-named author, Ruzsa and Schoen \cite{croot-ruzsa-schoen} that finds arithmetic progressions in single sumsets $A+B$, even when $A$ and $B$ are very sparse (much sparser than the sets considered in this paper).

Next we record a combination of Lemma \ref{local-sarkozy} and Proposition \ref{L_2-local} that will be useful to us in our proof of Roth's theorem. Recall that $[N] := \{1, \ldots, N\}$.

\begin{corollary}\label{main_L2_corollary}
Let $\delta \in (0,1)$ be a parameter and suppose that $A \subset [N]$ has size $\alpha N$, where $\alpha \geq 4N^{-\delta^2/36}$. Then there is a symmetric arithmetic progression $P \subset [-N/2,N/2]$ of length 
\[ |P| \geq \exp\left( \tfrac{1}{14} \left( \frac{\delta^2 \log{N}}{\log{4/\alpha}} \right)^{1/3} \right) \]
such that $0 \in P$ and, for each $t \in P$,
\[ \norm{ 1_A*1_A(x+t) - 1_A*1_A(x) }_2^2 \leq \delta^2 |A|^3. \]
\end{corollary}
\begin{proof}
Set
\begin{align*}
k &:= \left\lfloor \left(\frac{\delta^2 \log{N}}{36\log{4/\alpha}}\right)^{1/3} \right\rfloor, \\
\epsilon &:= \delta/k,
\end{align*}
and apply Proposition \ref{L_2-local} to $A$ with $B=A$ and $S = [N]$. Note that we may certainly take $K = 2/\alpha$ since $A + [N] \subset [2N]$. Thus we get a set $T \subset [N]$ of size at least $(\alpha/4)^{9/\epsilon^2} N$ that has 
\[ \norm{ 1_A*1_A(x+t) - 1_A*1_A(x) }_2^2 \leq \epsilon^2 |A|^3 \]
for each $t \in T-T$. Now apply Lemma \ref{local-sarkozy} to the set $T$ to get a symmetric arithmetic progression $P \subset kT-kT$ of length at least $2^{k+1}+1$. By the triangle inequality we then have that any $t \in P$ gives
\[ \norm{ 1_A*1_A(x+t) - 1_A*1_A(x) }_2^2 \leq \delta^2 |A|^3; \]
this progression would thus satisfy the conclusion of the corollary were it not for the fact that it may not be contained in $[-N/2,N/2]$. It is however contained in $[-kN,kN]$, and so we may simply select a symmetric subprogression $P' \subset [-N/2,N/2]$ of $P$ of length at least $2\lfloor 2^{k-1}/k \rfloor + 1$; this progression will then do. Note that the condition on $\alpha$ comes from the requirement that $k$ be at least $1$.
\end{proof}

\section{Arithmetic progressions in sumsets}\label{longAPs}

In this section we shall prove Theorem \ref{A+B}. Our task is thus to exhibit, for two sets $A$ and $B$ in $[N]$, the existence of a long arithmetic progression in the sumset $A+B$. We do this by combining Theorem \ref{AB-struct-full}---a consequence of Proposition \ref{L_p-local}---with Lemma \ref{local-sarkozy}.

\begin{proof}[Proof of Theorem \ref{A+B}]
Set 
\[ k := \left\lfloor \tfrac{1}{10} \left( \frac{\alpha \log{N}}{\log{4/\beta}} \right)^{1/4} \right\rfloor \]
and
\[ n := 2^{k+1}. \]
Assume $k \geq 1$, for otherwise the conclusion of the theorem is trivial. Apply Theorem \ref{AB-struct-full} to $A$ and $B$ with these parameters and $S = [N]$. Since 
\[ A+B \subset B + [N] \subset [2N] \]
we may certainly take $K_2 = 2/\alpha$ and $K_3 = 2/\beta$, and we may take $K_1 = K_2$ by \eqref{energy-CS}. This gives us a set $T \subset [N]$ of size $\delta N$, where
\[ \delta \geq \exp\left(-300 k^2 (\log{4/\beta}) (\log{2n}) /\alpha \right), \]
such that $A+B$ contains a translate of any subset $P$ of $kT-kT$ of size at most $n$. By Lemma \ref{local-sarkozy} we can find an arithmetic progression $P$ of length $n$ in $kT-kT$ provided 
\[ \delta \geq 2^{3k/2}/N^{1/(k+1)}, \]
a condition that may be seen to hold by a short calculation.
\end{proof}
\begin{remark}
In contrast to previous proofs of results of the form of Theorem \ref{A+B}, there was no need for us to embed the sets $A$ and $B$ in a finite group $\Zmod{p}$ for some prime $p$ larger than $N$ in order for us to carry out our analysis. Had we performed this embedding into $\Zmod{p}$, however, we would have been able to use a slight simplification of Lemma \ref{local-sarkozy}, since we would only need to use it for $K=1$.
\end{remark}
\begin{remark}
By a minor modification of the proof of Proposition \ref{L_p-local-left}---using the triangle inequality to get rid of the terms $1_A*1_B(x)$ instead of $\mu_C*1_B(x)$---one can deduce that $P \subset A+C$ for a very small set $C \subset B$. One thus needs to translate $A$ by very few elements of $B$ in order to generate long arithmetic progressions.
\end{remark}

We similarly get the following local version of Theorem \ref{A+B}.

\begin{theorem}[Arithmetic progressions in small sumsets]
Suppose $A$ and $B$ are finite, non-empty subsets of an abelian group such that 
\[ |A+B| \leq K_1 |A| \text{ and } |A+B| \leq K_2 |B|. \]
Then $A+B$ contains an arithmetic progression of length at least
\[ \tfrac{1}{2} \exp\left( c \left( \frac{\log{|A|}}{K_1 \log{2K_2}} \right)^{1/4} \right), \]
where $c > 0$ is an absolute constant.
\end{theorem}
\begin{proof}
This proof is virtually the same as that above. Set
\[ k := \left\lfloor \tfrac{1}{10} \left( \frac{\log{|A|}}{K_1 \log{2K_2}} \right)^{1/4} \right\rfloor \]
and $n := 2^{k+1}$. As before we apply Theorem \ref{AB-struct-full} to $A$ and $B$ with these parameters, but this time with $S = A$. Thus we get a set $T \subset A$ of size
\[ |T| \geq \exp\left(-150 k^2 K_1 (\log{2K_2}) (\log{2n}) \right) |A| \]
such that $A+B$ contains any subset of $kT-kT$ of size at most $n$. By the Pl\"unnecke-Ruzsa inequality, Theorem \ref{plunnecke-ruzsa}, we have that $|A+A| \leq K_1 K_2^2 |A|$. Another routine calculation now shows that we can apply Lemma \ref{local-sarkozy} to $T \subset A$ to find an arithmetic progression of length $n$ in $kT-kT$, which yields the result. (Note again that the theorem is trivial if $k < 1$.)
\end{proof}
\begin{remark}
Recall that arithmetic progressions may be degenerate in some groups; consider for example the group $\F_2^n$.
\end{remark}
\begin{remark}
Other local versions of this result are possible: we could for example work relative to a set $S$ of small doubling such that $|B+S| \leq K|B|$; this would yield slightly better bounds.
\end{remark}

We cannot mention this topic without drawing the reader's attention to a remarkable construction \cite{ruzsa:longAPs} of Ruzsa that places a limit on the potential strength of results of the above form:

\begin{theorem}
Let $\epsilon > 0$. For every prime $p > p_0(\epsilon)$ there is a symmetric set $A \subset \Zmod{p}$ of size at least $(1/2 - \epsilon)p$ such that $A+A$ contains no arithmetic progression of length 
\[ \exp\left( (\log{p})^{2/3+\epsilon} \right). \]
\end{theorem}

Let us also mention that if one only wishes to find arithmetic progressions of length about $\log N$ in $A+B$ then better results are available: one can work with much sparser sets than those considered in this paper by using the results in \cite{croot-ruzsa-schoen}.

\section{Roth's theorem}\label{roth}
In this section we give our proof of Theorem \ref{roths_theorem}. We shall employ a density-increment strategy, showing that if $A \subset \{1,\ldots,N\}$ is large and contains no three-term progressions then we can find a long arithmetic progression on which $A$ has significantly increased density. We can then iterate this argument in order to obtain a contradiction.

Let us introduce some notation before we begin. We denote the sum of a function $f : \Z \to \R$ with finite support over the three-term progressions in $\Z$ by $T_3(f)$; thus
\[ T_3(f) := \ssum_{x,y \in \Z} f(x)f(y)f(2y-x) = \ssum_{y} f(y)(f*f)(2y). \]
Note that we may drop parts of subscripts when the meaning is clear. If $f = 1_A$ is the indicator function of a set then $T_3(f)$ is simply the number of three-term progressions in $A$. Note that this includes trivial (constant) three-term progressions and that it counts $(x, x+d, x+2d)$ separately from $(x+2d, x+d, x)$. We shall use the notation $\mu_X$ to denote the normalized indicator function $1_X/|X|$ of a finite set $X$. For a subset $A$ of $X$ we shall say that $A$ has \emph{density $\alpha$ relative to $X$} if $|A| = \alpha |X|$; when $X$ is clear from the context we shall refer to $\alpha$ simply as the \emph{density} of $A$. Finally, we write $\E_{x \in X} f(x) = \frac{1}{|X|} \sum_{x \in X} f(x)$ for the average of $f$ over $X$.

The core of our proof of Roth's theorem lies in the following proposition.

\begin{proposition}\label{T3-approximation}
Let $\epsilon > 0$ and suppose that $A \subset [N]$ has size $\alpha N$. Then there is a symmetric arithmetic progression $P \subset [-N/8,N/8]$ of length at least 
\[ |P| \geq c \exp\left( c \left( \frac{\epsilon^2\log{N}}{\log{4/\alpha}} \right)^{1/3} \right), \]
where $c > 0$ is an absolute constant, such that
\[ | T_3(1_A * \mu_P) - T_3(1_A) | \leq \epsilon |A|^2. \]
\end{proposition}
\begin{proof}
Let $Q$ be the arithmetic progression obtained from Corollary \ref{main_L2_corollary} applied to $A$ with parameter $\epsilon^2$; thus $Q$ is large, $Q = -Q$ and $Q \subset [-N/2,N/2]$. Let $P$ be a symmetric subprogression of $Q$ of length at least $|Q|/8$ such that $4P \subset Q$; thus $P \subset [-N/8,N/8]$. We claim that this $P$ satisfies the conclusion of the proposition. Indeed,
\[ T_3(1_A*\mu_P) = \E_{(y,z,w) \in P^3} \ssum_x 1_A(x) 1_A*1_A(2x+2y-z-w) \]
and so 
\begin{align*}
| T_3(1_A*\mu_P) - T_3(1_A) | &= \left| \E_{y,z,w \in P} \ssum_x 1_A(x) \big( 1_A*1_A(2x+2y-z-w) - 1_A*1_A(2x) \big) \right| \\
&\leq |A|^{1/2}\ \E_{\substack{z,w \in P \\ y \in 2\cdot P}} \norm{1_A*1_A(x -y-z-w) - 1_A*1_A(x)}_2 \\
&\leq \epsilon |A|^2,
\end{align*}
these inequalities being instances of the triangle and Cauchy-Schwarz inequalities and the fact that $P+P+2\cdot P \subset Q$.
\end{proof}

We also require a preliminary lemma about $T_3$. The following lemma gives a lower bound for the minimal number of three-term progressions that a set (or a function) can contain given upper bounds on the function $r_3$; it is a quantitative version of an averaging argument of Varnavides \cite{varnavides}.

\begin{lemma}[Varnavides' theorem]\label{varnavides_thm}
Let $N$ be a positive integer and suppose that $f : [N] \to [0,1]$ is a function with average $\E_{x \in [N]} f(x) = \alpha$. Then, for any positive integer $M \leq N^{1/10}/2$,  
\[ T_3(f) \geq \left(\alpha - \tfrac{r_3(M)+2}{M}\right)M^{-4} N^2. \]
\end{lemma}

The proof of this lemma proceeds via a double-counting argument and can be found in \cite{croot-sisask:roth} for the case when $f$ is the indicator function of a set. In order to pass from a result about sets, like the lemma stated in \cite{croot-sisask:roth}, to a result about a function $f$ one can employ a standard probabilistic trick of defining a random set $A$ in $[N]$ by letting $x \in A$ with probability $f(x)$ independently for each $x$. See \cite[Exercise 10.1.7]{tao-vu} for more details.

We are now ready to proceed with the main body of the proof. We shall prove Theorem \ref{roths_theorem} in the following equivalent form.

\begin{theorem}
For any $c > 0$ there are positive numbers $C$ and $N_0$ such that
\[ r_3(N) \leq C N/(\log\log{N})^c \]
for all $N \geq N_0$.
\end{theorem}
\begin{proof}
We begin by establishing the theorem for some $c > 0$; we shall then be able to bootstrap this to establish the full result. Various inequalities in the argument will hold by the assumption that $N$ is large enough; we shall not state this assumption explicitly each time it is used.

Let $A$ be a subset of $\{1,\ldots, N\}$ of size $\alpha N = r_3(N)$ that does not contain any non-trivial three-term progressions, and let $\epsilon > 0$ be a parameter that is to be fixed later. Applying Proposition \ref{T3-approximation} to $A$ we obtain a long arithmetic progression $P$ such that 
\begin{align} | T_3(1_A * \mu_P) - T_3(1_A) | \leq \epsilon |A|^2. \label{T3-estimate}\end{align}
Our argument will be centred around the function
\[ 1_A*\mu_P(x) = |A \cap (x-P)|/|P|; \]
we shall show that if $0 < \delta < 1$ is chosen appropriately then there must be an $x$ for which \[ |A\cap(x-P)| > \delta^{-1}\alpha |P|. \]
This will form the base of our density increment argument.

Suppose, then, that $1_A*\mu_P(x) \leq \delta^{-1} \alpha$ for all $x \in \Z$. Let $f(x) := (\delta/\alpha)1_A*\mu_P(x)$, so that $0 \leq f(x) \leq 1$ for all $x$, $\ssum_x f(x) = \delta N$, and 
\[ T_3(f) = (\delta/\alpha)^3 T_3(1_A*\mu_P). \]
Note also that $f$ is supported on $A+P \subset [1-N/8, 9N/8] \cap \Z$, an interval of size at most $5N/4$. Now, $A$ contains only trivial three-term progressions and so $T_3(1_A) = |A|$. Thus \eqref{T3-estimate} implies that
\begin{align}
T_3(f) \leq 2 \delta^3\epsilon N^2/\alpha \label{T3(f)_upper-bound}
\end{align}
provided $\epsilon \geq 1/|A|$. On the other hand, Lemma \ref{varnavides_thm} tells us that 
\[ T_3(f) \geq \left(\tfrac{4}{5}\delta - \tfrac{r_3(M) + 2}{M}\right) M^{-4} N^2 \]
provided $M \leq N^{1/10}/2$.

Let us initially pick $\delta = 9/10$. One may check by hand that $r_3(10) = 5$; by picking $M = 10$ we therefore see that $T_3(f) \geq c_0 N^2$ for some positive absolute constant $c_0$. Comparing this to \eqref{T3(f)_upper-bound} we see that we obtain a contradiction provided we pick $\epsilon = c_1 \alpha$ for some small constant $c_1 > 0$. (This is permissible provided $\alpha \geq 1/\sqrt{c_1 N}$, which we assume.) Hence we must have that 
\[ |A \cap (x-P)| \geq \tfrac{10}{9} \alpha |P| \]
for some integer $x$, where $P$ is a rather long progression. Let us assume that $\alpha \geq (\log N)^{-1/6}$. Then 
\[ |P| \geq \exp\left( (\log N)^{1/8} \right); \]
we have thus shown that $A$ has density at least $\tfrac{10}{9} \left(\tfrac{r_3(N)}{N}\right)$ on an arithmetic progression of length $N_1 := |P|$. We may thus rescale to obtain a set $A_1 \subset \{1,\ldots, N_1\}$ that is also free of arithmetic progressions, but that is now much denser than the original set $A$.

We may now iterate this argument, obtaining a sequence of integers $N_j$ with 
\[ N_j \geq \exp\left( (\log N_{j-1})^{1/8} \right) \]
and a sequence of densities $\delta_j$ such that 
\[ \delta_j \geq \left(\tfrac{10}{9}\right)^j \left(\tfrac{r_3(N)}{N}\right), \]
the only requirements for proceeding to the next stage of the iteration being that $\delta_j \geq (\log N_j)^{-1/6}$ and $N_j \geq C$ for some absolute constant $C$. Since no $\delta_j$ can exceed $1$, this iteration must stop at some stage $K$ with $K \leq \tfrac{\log(N/r_3(N))}{\log 10/9}$, at which point one of these requirements must fail. From this we may deduce that
\[ r_3(N) \leq \frac{C N}{(\log\log N)^{\frac{\log 10/9}{\log 8}}} \]
for some absolute constant $C$.

This proves the theorem for a fixed exponent $c$ of $\log\log N$. We may now use this to run the argument again, except that we do not now need to rely on numerical data in order to apply Lemma \ref{varnavides_thm} effectively. That is, we may now pick $\delta$ arbitrarily small and then find a fixed value $M$ for which $\tfrac{4}{5}\delta - \tfrac{r_3(M)+2}{M} \geq \delta/2$. This means that, instead of obtaining a density increment of a factor of $\tfrac{10}{9}$, we may obtain an increment of an arbitrarily large factor $\delta^{-1}$, still on a progression of length at least $\exp\left( (\log N)^{1/8} \right)$ (though we now need $N$ to be large enough in terms of $\delta$). Following the above argument through again, this shows that 
\[ r_3(N) \leq \frac{C N}{(\log\log N)^{\frac{\log 1/\delta}{\log 8}}} \]
for $N \geq N_0(\delta)$ and some constant $C$ depending on $\delta$.
\end{proof}

\section{Strong approximate groups}\label{discts}
Finally we prove Proposition \ref{discontinuous}, the result about strong approximate groups; recall that we say that $A$ is a strong $K$-approximate group if $1_A*1_A(x) \geq |A|/K$ for each $x \in A^2$. This proposition does not follow directly from the almost-periodicity results; instead it uses the ideas in the proofs of those results in a slightly different way.

\begin{proof}[Proof of Proposition \ref{discontinuous}]
We shall show that if $C \subset A$ is chosen at random then $CA \approx A^2$ with good probability. Indeed, let us start by picking a random set $C \subset A$ of size $k$. By the hypothesis on $A$, any $x \in A^2$ that satisfies $|\mu_C*1_A(x) - 1_A*1_A(x)| < |A|/K$ lies in $CA$, whence
\[ \P( x \notin CA ) \leq \P\left( |\mu_C*1_A(x) - 1_A*1_A(x)| \geq |A|/K \right) \leq 2e^{-2k/K^2}, \]
the latter inequality being a standard distributional inequality for hypergeometric distributions; see, for example, \cite{chvatal} (and cf. Proposition \ref{binomial_deviations}). Summing this over all $x \in A^2$ we obtain the estimate
\[ \E |\{ x \in A^2 : x \notin CA \}| \leq 2e^{-2k/K^2} |A^2|. \]
Markov's inequality therefore yields
\[ \P\left( |A^2 \bigtriangleup CA| \leq \lambda |A^2| \right) \geq 1 - 2e^{-2k/K^2}/\lambda; \]
let us pick $\lambda := 4e^{-2k/K^2}$ to make this probability be at least $1/2$.

Now note that $|A^2| \leq K|A|$; this follows from the inequality $1_A*1_A(x) \geq |A| 1_{A^2}(x)/K$ holding for all $x$. As in the proof of Proposition \ref{L_2-local-left}, this means that there is a set $C$ and a set $T \subset A^{-1}$ of size at least $|A|/(2K)^k$ such that 
\[ |A^2 \bigtriangleup tCA| \leq \lambda |A^2| \]
for any $t \in T$. For any two elements $t_1, t_2 \in T$ we therefore have
\[ | t_2 t_1^{-1} A^2 \bigtriangleup A^2 | \leq 2\lambda |A^2| \]
by the triangle inequality. Thus we may take $S := T T^{-1}$ after choosing $k := \left\lceil (K^2\log{8/\epsilon})/2 \right\rceil$.
\end{proof}

\begin{remark}
It is easy to see that a strong $K$-approximate group must have small doubling, $|A^2| \leq K|A|$, but unlike with sets of small doubling it is not clear how abundant strong $K$-approximate groups of different sizes are, even in the group $\Zmod{p}$ for a prime $p$. Konyagin \cite[Problem 5]{lev:rep-functions} raised the basic question of whether it is the case that for any set $A \subset \Zmod{p}$ of size at most $\sqrt{p}$ there exists some element $x \in A+A$ such that $1_A*1_A(x) \leq C |A|^{1-c}$, where $C, c>0$ are absolute constants. Partial progress was made on this question by \L uczak and Schoen \cite{luczak-schoen}, who also noted that work of Green and Ruzsa \cite{green-ruzsa:rectification} implies that one can always find an $x \in A+A$ with $1_A*1_A(x) \leq \max(1, |A|/(\log_2 p)^{1/2+o(1)})$. The results of this paper can be used to derive a bound similar to this, if perhaps slightly stronger, but we do not pursue this here.
\end{remark}

\section{Further remarks}\label{remarks}
We conclude with some remarks.

\subsection{Convolutions of functions}
Although we have focused on convolutions of sets in this paper, it is relatively easy to deduce results for convolutions of functions. Indeed, let $f, g : G \to [0,1]$ be two functions with finite supports $S_f$ and $S_g$. Define random sets $A, B \subset G$ by stipulating that $x \in A$ with probability $f(x)$ and $x \in B$ with probability $g(x)$, all independently. One may then use a concentration inequality such as Chernoff's inequality \cite[Theorem 1.8]{tao-vu} to deduce that there is a choice of sets $A \subset S_f$ and $B \subset S_g$ such that $A$ has size very close to $\ssum f$, $B$ has size very close to $\ssum g$ and $|1_A*1_B(x) - f*g(x)|$ is small for every $x \in S_1+S_2$. An almost-periodicity result saying that
\[ \norm{ f*g(tx) - f*g(x) }_2^2 \leq \epsilon^2 \left(\ssum f\right) \left(\ssum g \right)^2 \]
for every $t \in TT^{-1}$ for a large set $T$ then follows from the corresponding result for sets, and similarly for $L^p$-almost-periodicity. One may then deal with arbitrary real-valued functions with finite support by rescaling. It is also likely that one can prove the almost-periodicity results directly for functions, though the statements will look slightly different; we do not pursue this here.

\subsection{Comparisons with Fourier-analytic results}
Our proofs of the almost-periodicity results in this paper have been combinatorial, which meant that there was no need for us to distinguish between abelian and non-abelian groups. When dealing with finite abelian groups, however, it is possible to derive results similar to Corollaries \ref{L_2-global} and \ref{L_p-global} using Fourier analysis. Indeed, in the abelian setting Corollary \ref{L_2-global} is essentially a result of Bogolyubov \cite{bogolyubov} coupled with a result of Chang \cite{chang} on the large spectra of subsets of abelian groups; see Lemma 4.36 and (the proof of) Proposition 4.39 in \cite{tao-vu}. An important difference between the two approaches is that Fourier analysis provides one with more information about the set $T$: one may take it to be a so-called Bohr set (an approximate annihilator of a set of characters in the Pontryagin dual of $G$), and it is well known that Bohr sets are arithmetically structured sets. For instance, Bohr sets contain long arithmetic progressions, which means that one does not need to appeal to structure-generation results like Lemma \ref{local-sarkozy}. If one uses this as the base for the arguments of \S\ref{roth} (set in $\Zmod{N}$ rather than $[N]$) then one can obtain a bound for $r_3(N)$ similar to that of an old but recently published proof of Roth's theorem due to Szemer\'edi \cite{szemeredi}; indeed, our argument is in some ways quite similar to Szemer\'edi's. We present further details of this argument in the note \cite{croot-sisask:SzemerediRothNote}.

It is much less clear that one can obtain an $L^p$-almost-periodicity result of a type similar to Corollary \ref{L_p-global} for abelian groups using Fourier analysis. One may extract such a result from the paper \cite{bourgain:longAPs} of Bourgain that exhibits the existence of long arithmetic progressions in $A+B$; indeed, the main thrust of the paper is to establish the estimates required to prove such an almost-periodicity result. Specifically one can obtain a result of the following type.

\begin{proposition}\label{bourgain-periodicity}
Let $G$ be a finite abelian group and let $\epsilon > 0$ and $m \in \N$ be two parameters. Suppose that $f, g : G \to [0,1]$ have averages $\E_{x \in G} f(x) = \alpha$ and $\E_{x \in G} g(x) = \beta$. Then there is a Bohr set $B = B(\Gamma, \rho)$ of rank $|\Gamma| \ll m^2\log(1/\epsilon)/\epsilon^2$ and radius $\rho = c \epsilon^3/m$ such that
\[ \norm{ f*g(x+t) - f*g(x) }_{2m} \leq \epsilon (\alpha \beta)^{1/2}|G|^{1+1/2m} \]
for each $t \in B$.
\end{proposition}

By a Bohr set $B(\Gamma, \rho)$ here we mean a set of the form 
\[ \{ x \in G : |\gamma(x)-1| \leq \rho \text{ for all $\gamma \in \Gamma$} \}, \]
where $\Gamma \subset \widehat{G}$ is a collection of characters.

Bourgain's argument is very elegant though also somewhat complex, relying on some quite sophisticated manipulations of sets of Fourier coefficients. We shall not say more about this here, save for making two comments. First, the set $B$ produced by the above proposition will in general be somewhat smaller than the set $T$ given by Corollary \ref{L_p-global}, but is also guaranteed to contain more structure, which is ultimately what yields Bourgain's superior exponent of $1/3$ in place of our $1/4$ in the length of the arithmetic progressions one finds in $A+B$. Second, if one wishes to compare the $L^p$ norm to $\alpha \beta$, say, then Corollary \ref{L_p-global} is useful even if one of the sets $A$ and $B$ is rather sparse whereas Proposition \ref{bourgain-periodicity} requires both sets to be quite large. More details about Proposition \ref{bourgain-periodicity} may be found in the note \cite{sisask:bourgain-APs}.

Obtaining the local versions of our results using Fourier analysis seems harder. We note that there are tools that get around this to some extent; notably there is the `modelling' lemma of Green and Ruzsa \cite{green-ruzsa:freiman} that allows one to `isomorphically embed' a set $A \subset G$ with small doubling $|A+A| \leq K|A|$ as a dense set $A' \subset G'$, where $|A'| \geq f(K) |G'|$. See for example the paper \cite{sanders:local-roth} of Sanders for an efficient proof of a local version of Roth's theorem that makes use of this lemma. Interestingly, modelling results of the same kind cannot exist for non-abelian groups \cite{tao:no-finite-models}.

\subsection{Roth's theorem in other settings}
In this paper we proved Roth's theorem in the setting of the integers $\{1, \ldots, N\}$. The Fourier-analytic proofs of Roth's theorem generally become simpler when studied in the vector space $\F_3^n$ over the finite field $\F_3$ (or $\F_p^n$ for a fixed prime $p$), and this holds true for our argument as well. There are two main reasons for this. One is that it is very easy to establish a result similar to Lemma \ref{local-sarkozy} in $\F_3^n$, as remarked in \S\ref{structure-generation}. The other is that it becomes easier to run through the density increment strategy itself, since one can induct on subspaces rather than on arithmetic progressions. In particular one does not really need a result corresponding to Lemma \ref{varnavides_thm} (Varnavides' theorem). The bounds one obtains for $r_3(\F_3^n)$ are not significantly better than the corresponding ones for $r_3(N)$ with $N \approx 3^n$, however.

We should mention in this context that Seva Lev has recently produced a proof \cite{lev:meshulam} of the $\F_p^n$-version of Roth's theorem that removes the use of characters from the general framework of Meshulam's proof \cite{meshulam}. Lev's proof involves very different ideas to those of this paper, however.

\subsection{Extensions}
There are many possible potential extensions of the methods presented in this paper. It seems likely that the ideas used could also be used to tackle locally compact groups, this being a natural setting for many of the results considered here (where we have only dealt with discrete groups). An area of application that we have not discussed in detail in the current paper is that of Freiman-type results; let us for now remark that it is easy to obtain a number of rudimentary Freiman-type theorems by coupling our almost-periodicity results with so-called covering lemmas. This will be followed up elsewhere.

\appendix
\section{The moments of the binomial and hypergeometric distributions}\label{appendix:hyper_moments}
As noted in the proof of Proposition \ref{L_p-local-left}, if one selects a random $k$-element subset $C$ from a set $A$ in an ambient group $G$ then, for any fixed element $x \in G$, the random variable $1_C*1_B(x)$ follows a hypergeometric distribution. In this appendix we prove the bounds of Lemma \ref{hyper_moments} on the moments of such a distribution.

Recall that $X$ follows a hypergeometric distribution with parameters $N$, $M$ and $k$ if
\[ \P(X = j) = \binom{M}{j}\binom{N-M}{k-j}\Big/\binom{N}{k}, \]
so that $X$ can be thought of as counting the number of marked objects selected when $k$ objects are picked without replacement from a population of $N$ objects, $M$ of which are marked. If the $k$ objects are selected \emph{with} replacement then the number of marked objects selected follows a binomial distribution with parameters $n = k$ and $p = M/N$, and the two distributions are closely related. We have found certain estimates for the binomial distribution to be more readily available in print than the corresponding estimates for the hypergeometric distribution; the following corollary of a result of Hoeffding \cite[Theorem 4]{hoeffding} allows us to make use of these results.

\begin{proposition}
Let $X$ follow a hypergeometric distribution as above and let $Y$ follow a binomial distribution with parameters $n = k$ and $p = M/N$. Then for any convex, continuous function $f$ we have 
\[ \E f(X) \leq \E f(Y). \]
In particular, for $m \geq 1/2$ we have
\[ \E \abs{ X - \tfrac{k M}{N} }^{2m} \leq \E \abs{ Y - np }^{2m}. \]
\end{proposition}

Lemma \ref{hyper_moments} therefore follows immediately from the following proposition.

\begin{proposition}\label{binomial_moments}
Let $m \geq 1$ and suppose that $X$ follows a binomial distribution with parameters $n$ and $p$. Then
\begin{equation}
\E |X - np|^{2m} \leq 2(3m n p + m^2)^m. \label{binom_conditional_estimate}
\end{equation}
\end{proposition}

In order to prove this we shall make use of the following deviation estimates, the type of which is often associated with the names of Bennett, Bernstein, Chernoff and Hoeffding.

\begin{proposition}\label{binomial_deviations}
Let $X$ follow a binomial distribution with parameters $n$ and $p$. Then
\begin{align}
\P(X \leq np - t) &\leq \exp\left(-\frac{t^2}{2np}\right) \label{binom_small} \\
\text{and} \quad \P(X \geq np + t) &\leq \exp\left(-\frac{t^2}{2(np + t/3)}\right) \label{binom_large}
\end{align}
for any $t \geq 0$.
\end{proposition}

Proofs of these estimates may be found in \cite{janson}; see also \cite{bollobas:random_graphs} and \cite{alon-spencer}. They can be derived from an application of Markov's inequality to the random variable $e^{\lambda(X-np)}$ using the fact that the moment generating function $\E e^{\lambda(X - np)}$ is $e^{-\lambda p n} (p e^{\lambda} + 1-p)^n$.

\begin{proof}[Proof of Proposition \ref{binomial_moments}]
We may write 
\begin{equation}
\E |X - np|^{2m} = \int_0^{\infty} \P( |X - np|^{2m} > t ) \ud t. \label{moment-integral}
\end{equation}
Since $\P( |X - np| > t ) = \P( X < np - t ) + \P( X > np + t )$ we may decompose the right-hand side of \eqref{moment-integral} as a sum of two integrals $I^-$ and $I^+$ in an obvious way. The deviation estimates \eqref{binom_small} and \eqref{binom_large} then give
\[ I^- \leq \int_0^{\infty} \exp\left(-\frac{t^{1/m}}{2np}\right) \ud t = (2np)^m \Gamma(m+1) \]
and
\[ I^+ \leq \int_0^{\infty} \exp\left( \frac{-t^{1/m}}{2(np + t^{1/2m}/3)} \right) \ud t. \]
We split the range of integration of this latter integral into two parts $I_1$ and $I_2$ defined as follows. Let $\lambda := \frac{1}{3} + \frac{1}{3}\sqrt{1+6np/m}$, so that $9(\lambda m)^2/2(np + \lambda m)=3m$; $I_1$ is then the integral over the range $0 \leq t \leq (3\lambda m)^{2m}$ and $I_2$ the integral over the remaining range. Thus
\[ I_1 \leq \int_0^{\infty} \exp\left( -\frac{t^{1/m}}{2(np+\lambda m)}\right) \ud t = \left(2np + 2\lambda m\right)^m \Gamma(m+1). \]
We need to take a little more care with $I_2$. Let us write $w := \frac{9(\lambda m)^2}{2(np + \lambda m)} = 3m$. Then
\[ I_2 \leq \int_{(3\lambda m)^{2m}}^{\infty} \exp\left(- \frac{3t^{1/2m}}{2(1+np/\lambda m)}\right) \ud t = 2m\, \left(\frac{2(np+\lambda m)}{3\lambda m}\right)^{2m} \int_w^{\infty} z^{2m-1} e^{-z} \ud z. \]
Making the change of variables $u = z - w$, this last integral becomes
\[ w^{2m-1} e^{-w} \int_0^{\infty} \left(1+\tfrac{u}{w}\right)^{2m-1} e^{-u} \ud u \leq w^{2m-1} e^{-w} \int_0^{\infty} e^{-u(1-2m/w)} \ud u, \]
the inequality holding since $1+x \leq e^x$ for all $x$, and this expression equals $w^{2m} e^{-w}/m$. Thus
\[ I_2 \leq 2(3\lambda m)^{2m} e^{-3m}. \]

Combining these estimates for $I^-$ and $I^+ = I_1 + I_2$ we obtain 
\begin{align*}
\E |X - np|^{2m} &\leq (2np)^m \Gamma(m+1) + (2np+2\lambda m)^m \Gamma(m+1) + 2(9\lambda^2 m^2/e^3)^m.
\end{align*}
Using the easily-verifiable bound $\Gamma(m+1) \leq 2(3m/5)^m$ and the definition of $\lambda$ then yields \eqref{binom_conditional_estimate} after some routine but technical calculations.
\end{proof}

\begin{remark}
By being a bit more careful in the above proof one could obtain somewhat smaller values for the constants appearing in the proposition, though this is not particularly important for our applications. We should also remark that, although we only required it for binomial random variables, Proposition \ref{binomial_moments} holds even when $X$ is a sum of independent Bernoulli random variables that are not necessarily identically distributed. In that setting $n$ is the number of summands and $p$ is $\E X/n$, and one may prove the result exactly as above since Proposition \ref{binomial_deviations} holds for such random variables. (Let us also note that Proposition \ref{binomial_moments} holds with different constants for sums of more general random variables.)
\end{remark}

\end{document}